\documentclass[11pt, a4paper]{article}

\usepackage{amsmath}
\usepackage{amsfonts}
\usepackage{amssymb}
\usepackage{amsthm}
\usepackage[ansinew]{inputenc}
\usepackage{alltt}
\usepackage{color}

\definecolor{string}{rgb}{0.7,0.0,0.0}
\definecolor{comment}{rgb}{0.13,0.54,0.13}
\definecolor{keyword}{rgb}{0.0,0.0,1.0}

\fussy

\theoremstyle{plain}
\newtheorem{theorem}{Theorem}[section]
\newtheorem{proposition}[theorem]{Proposition}
\newtheorem{lemma}[theorem]{Lemma}
\newtheorem{definition}[theorem]{Definition}

\theoremstyle{definition}

\newtheorem*{notation}{\sc Notations}

\newtheorem{corollary}[theorem]{Corollary}

\newtheorem{assumption}[theorem]{Assumption}
\newtheorem{remark}[theorem]{Remark}

\def\beq{\begin{equation}\displaystyle}
\def\eeq{\end{equation}}

\renewcommand\appendix{\bigskip {\noindent \Large \bf Appendix : Technical Lemmas for the diffusive limit}
  \setcounter{section}{0}%
  \setcounter{subsection}{0}%
\setcounter{equation}{0}%
\setcounter{theorem}{0}%
\def\thetheorem{A.\arabic{theorem}}
\def\theequation {A.\arabic{equation}}}

\newcommand{\dsp}{\displaystyle}
\newcommand{\calU}{\mathcal U}
\newcommand{\calL}{\mathcal L}
\newcommand{\calZ}{\mathcal Z}
\newcommand{\calE}{\mathcal E}
\newcommand{\RR}{\mathbb{R}}
\newcommand{\NN}{\mathbb{N}}
\newcommand{\intdouble}{\int\mbox{\hspace{-3mm}}\int}
\newcommand{\inttriple}{\int\mbox{\hspace{-3mm}}\int\mbox{\hspace{-3mm}}\int}

\makeatletter
\renewcommand\theequation{\thesection.\arabic{equation}}
\@addtoreset{equation}{section}
\makeatother

\title{Analysis of a diffusive effective mass model for nanowires}
\author{C. Jourdana\footnote{Istituto di Matematica Applicata 
e Technologie Informatiche, CNR, Via Ferrata 1, 27100 Pavia, Italy.
clement@imati.cnr.it} \footnote{Institut de Math\'ematiques de Toulouse, Univ Paul Sabatier, 118 Route de Narbone, 31062 Toulouse, France}  
\ and N. Vauchelet\footnote{UPMC Univ Paris 06, UMR 7598, Laboratoire Jacques-Louis Lions, F-75005, Paris, France.
vauchelet@ann.jussieu.fr}}
\date{}

\begin{document}

\maketitle

\begin{center}
{\large\it This work is dedicated to Naoufel Ben Abdallah, who was\\ a talented researcher, an enthusiastic supervisor and a generous person.}
\end{center}

\begin{abstract}
We propose in this paper to derive and analyze a self-consistent model describing the diffusive transport in a nanowire. From a physical point of view, it describes the electron transport in an ultra-scaled confined structure, taking in account the interactions of charged particles with phonons. The transport direction is assumed to be large compared to the wire section and is described by a drift-diffusion equation including effective quantities computed from a Bloch problem in the crystal lattice.
The electrostatic potential solves a Poisson equation where the particle 
density couples on each energy band a two dimensional confinement density with the monodimensional transport density given by the Boltzmann statistics.
On the one hand, we study the derivation of this Nanowire Drift-Diffusion Poisson model from a kinetic level description. On the other hand, we present an existence result for this model in a bounded domain.
\end{abstract}

\section{Introduction}

A quantum wire is an electronic component made of a periodic ion packing. The transport direction is large compared to the wire section, which includes only few atoms. So, the assumption of infinite periodic structure in the wire cross section, which allows to derive the usual effective mass theorem \cite{ashmer,wen}, cannot be used anymore.

In \cite{nbajourpie}, a new quantum model for nanowires is derived. Using an envelope function decomposition, \cite{nbabarletti} is extended to nanowires and a longitudinal effective mass model is obtained. However, in many applications such that FETs (Field Effect Transistors) for example, semiconductor devices contains largely doped regions. In these regions, collisions play an important role in the transport. Usually, quantum models do not include collisions of charged particles. That is the reason why a diffusive model has to be developed.

In this paper, a Nanowire Drift-Diffusion Poisson NDDP model is derived following \cite{ddsp} by performing a diffusive limit from a sequence of 1d Boltzmann equations in the transport direction, one for each energy band. Similarly to \cite{nbajourpie}, this model takes into account the ultra-scaled confinement and retains information of the nanowire cross section. Moreover, a self-consistent model includes the resolution of the Poisson equation in the entire device.

\subsection{Nanowire quantities}

In order to define the effective masses and the other physical quantities used in the NDDP model, we need to recall the model derived in \cite{nbajourpie}. We consider  an infinite wire defined in a physical domain $\mathbb{R}\times\omega_{z_{a}}$, where $a$ is the typical spacing between lattice sites. The transport is described by a scaled Schr\"odinger equation in $\mathbb{R}\times\omega_{z_{a}}$ containing a potential $W_{\mathcal{L}}$ generated by the crystal lattice, fast oscillating in the scale defined by the crystal spacing, and a slowly varying potential $V$ computed self-consistently through the resolution of a Poisson equation in the whole domain. Since the 2d cross section $\omega_{z_{a}}$ comprises few ions, $W_{\mathcal{L}}$ is considered periodic only in the longitudinal $x$-direction, also called transport direction. The variable $z$ of the transverse section can be considered as fast variable, and it can be rescaled as $z'=\frac{z}{a}$. 
To simplify notations, we now omit the primes. Then, $\omega_z$ will denote the scaled cross section and we assume to work in rescaled quantities such that the periodicity is setted to $1$ in the transport direction.

The starting point is the definition of the generalized Bloch functions as the eigenfunctions of the following problem in the 3d unit cell $\mathcal U=(-1/2,1/2)\times \omega_z$ :
\begin {equation}\label{eigen_value_pb}
\left \{
\begin{array}{c @{\hspace{1.3cm}} c}
 -\frac{1}{2} \Delta \chi_{n} + W_{\mathcal{L}} \chi_{n} = E_n \chi_{n}.&\vspace{0.2cm}\\
 \chi_{n}(y,z)=0 ~~ \mbox{on }\partial\omega_z, \quad
 \chi_{n} \hbox{ 1-periodic in }y.&\vspace{0.2cm}\\
\int_{\mathcal U} |\chi_n|^2 dy dz=1.&
\end{array}
\right.
\end{equation}
We use here the notation $y$ to emphasize we consider only one period $(-1/2,1/2)$. We point out that the boundary conditions are representative of our nanowire problem. Indeed, we consider the periodicity only in the transport direction and we choose homogeneous Dirichlet conditions in other directions in order to impose confinement. 

\begin{assumption}\label{hyp_WL}
We assume that $W_{\mathcal{L}}$ is a nonnegative potential given in
$L^\infty(\calU)$.
\end{assumption}
Under Assumption \ref{hyp_WL}, verified by physically relevant potentials, it is well-known that the
eigenfunctions $\chi_n$, solutions of (\ref{eigen_value_pb}), form an orthonormal basis of $L^2(\mathcal U)$ (see e.g. \cite{henrot}), with real eigenvalues 
which satisfies
\begin{equation}
E_1 \leq E_2 \leq ...,\hspace{1cm} \lim_{n\rightarrow +\infty} E_n = +\infty.
\end{equation}

\begin{assumption}\label{simple_E}
We assume that the eigenvalues $E_n$ are all simple.
\end{assumption}
Once we compute the $\chi_n$'s, we can define $m_n^*$ which corresponds to the $n^{th}$ band effective mass and which is given by
\begin{equation}\label{effective_mass}
\frac{1}{m_n^*}=1-2\sum_{n'\neq n} \frac{P_{nn'}P_{n'n}}{E_n-E_{n'}},
\end{equation}
where
\begin{equation}\label{p_nn}
P_{nn'}=\int_{\mathcal{U}}\partial_y \chi_{n'} (y,z) \chi_n(y,z)\, dy dz
\end{equation}
are the matrix elements of the gradient operator between Bloch waves. We also define the quantities
\begin{equation}\label{v_nn}
V_{nn}(t,x)=\int_{\omega_z} V(t,x,z) g_{nn}(z)\, dz
=\langle V(t,x,\cdot),g_{nn}\rangle,
\end{equation}
with
\begin{equation}\label{g_nn}
g_{nn}(z)=\int_{-1/2}^{1/2} \chi_n^{2}(y,z)\,dy,
\end{equation}
where we make use of the notation 
$\langle f,g\rangle = \int_{\omega_z}f(z)g(z)\,dz$. We notice that $g_{nn}$'s are quantities that contain information of the cross section and they allow to link the one dimensional transport direction to the entire nanowire. They play an important role not only in \cite{nbajourpie} but also in the NDDP model presented in this paper.

Finally, the fully quantum longitudinal effective mass model obtained in \cite{nbajourpie} consists of a sequence of one dimensional device dependent Schr\"odinger equations, one for each band :
$$
\imath \partial_t \psi_n (t,x)= -\frac{1}{2m_n^*} \partial_{xx} \psi_n(t,x) + V_{nn}(t,x) \psi_n(t,x).
$$
The physical parameters, $m_n^*$ and $g_{nn}$, computed once for a given
device by solving the generalized Bloch problem (\ref{eigen_value_pb}), are integrated in transport equations. Next, self-consistent computations include the resolution of the Poisson equation.

Assumption \ref{simple_E} is restrictive. In \cite{nbajourpie}, a more plausible assumption is discussed, related to symmetry properties of the crystal.
Then, to each multiple eigenvalue corresponds a system of coupled Schr\"odinger equations with dimension equal to the multiplicity of the eigenvalue. The kinetic part of the effective mass Hamiltonian is diagonal and the coupling occurs through the potential. Nevertheless, the derivation of a diffusive model in this case is far from the scope of this paper and is not discussed here.

%


\subsection{Diffusive transport description}

In this paper, we mainly consider a finite wire in the transport direction $x$
defined in the bounded domain $\omega_x=[0,L]$ such that $L>>a$. We denote $\Omega=[0,L]\times\omega_z$ this bounded device.
Since $L$ is large compared to $a$, the crystal lattice can be assumed periodic only in the transport direction as presented in the above subsection. Moreover, we will consider that the evolution
of charged particles is mainly driven by collisions with phonons which 
represent lattice vibrations. A widely used model to describe such kind
of transport in various area such as plasmas or semiconductors is the 
drift-diffusion equation. It consists in a conservation equation 
of the particle density in the transport direction which is called 
here the surface density $N_s(t,x)$ and which corresponds to the integral 
in the direction $z$ of the total density. 
The current is the sum of a drift term and of a diffusion term 
\cite{jungel,mrs,mock}.
The equation reads
\begin{equation}\label{driduf}
\partial_t N_s -\partial_x\Big(D(\partial_x N_s + N_s \partial_x V_s) \Big)=0,
\end{equation}
where D is a diffusion coefficient and $V_s(t,x)$ is the effective potential.
This potential is self-consistant and takes into account the quantum 
confinement in the nanowire. Its derivation will be specified in 
Section \ref{derivation} ; in particular, we will show that its 
expression is given by
\begin{equation}\label{eq_vs}
V_s=-\log \mathcal{Z}\hspace{0.5cm}\hbox{with}\hspace{0.5cm} \mathcal{Z}=\sum_{n=1}^{+\infty}e^{-\big(E_n+V_{nn}\big)},
\end{equation}
where $E_n$ are the eigenvalues of the problem (\ref{eigen_value_pb}) 
and $V_{nn}(t,x)$ are the potential energies defined by (\ref{v_nn}).
It is also usual to introduce the Fermi level $E_F(t,x)$ and the 
Slotboom variable $u(t,x)$ defined by
\begin{equation}\label{fermilevel}
u=e^{E_F}=N_S e^{V_S}=\frac{N_s}{\mathcal{Z}}.
\end{equation}
Then, the current $J(t,x)$ can be expressed as
$$
J=D(\partial_x N_s + N_s \partial_x V_s)=D N_s \partial_x E_F = D e^{-V_s} \partial_x u.
$$

The electrostatic potential $V(t,x,z)$ is solution of the Poisson equation 
\begin{equation}\label{pois}
-\Delta_{x,z} V = \rho. 
\end{equation}
The three dimensional macroscopic charge density $\rho(t,x,z)$ takes into account the 
contribution of all energy bands ; it is defined as follows
\begin{equation}
\rho=\sum_{n=1}^{+\infty} N_n g_{nn},
\end{equation}
where $g_{nn}(z)$ is given by (\ref{g_nn}) and $N_n(t,x)$ is the 
charge density in the transport direction which is expressed in the 
approximation of Boltzmann statistics by
\begin{equation}\label{density}
N_n=e^{E_F-\big(E_n+V_{nn}\big)}=\frac{N_s}{\mathcal{Z}}e^{-\big(E_n+V_{nn}\big)}.
\end{equation}
We point out that the link between the one dimensional densities $N_n$ and the charge density $\rho$ is done using the $g_{nn}$'s, as it is justified in \cite{nbajourpie}. 
We complete this system with the following boundary conditions
\begin{eqnarray}
\label{bdy1}
N_s(t,x)=N_b \hspace{1cm} &\hbox{for }& x\in \partial\omega_x,\\ 
\label{bdy2}
V(t,x,z)=V_b(z) \hspace{1cm} &\hbox{for }& x\in \partial\omega_x,\\
\label{bdy3}
\partial_z V(t,x,z)=0 \hspace{1cm} &\hbox{for }& z\in \partial\omega_z.
\end{eqnarray}
These boundary conditions do not correspond to the mixed type boundary conditions necessary for physical applications (taking in account source and drain contacts, gate...). It is chosen for the mathematical convenience and in particular for the elliptic regularity properties of the Poisson equation \eqref{pois} on our domain.

To simplify notations, we define the functional $\mathcal{S}[V](t,x,z)$ such that
\begin{equation}\label{def_S}
\mathcal{S}[V]= \sum_{n=1}^{+\infty} \frac{e^{-\big(E_n+V_{nn}\big)}}{\mathcal{Z}} g_{nn}.
\end{equation}
With this notation, we have
\begin{equation}
\rho= N_s \mathcal{S}[V].
\end{equation}

\subsection{Main results}

In \cite{ddsp,heitzring,ddsplog}, the authors propose transport models for
confined structures using the subband description which allows to reduce
the 3d problem to a 2d transport equation.
The transport coefficients have then to be computed by solving eigenvalue
problems for the steady-state Schr\"odinger equation in the confinement
direction, which is therefore one dimensional.
Compared to \cite{ddsp,ddsplog}, the NDDP model presented in this paper,
involved the resolution of the Bloch problem in all directions and the 
confinement is two-dimensional.

The main results in this paper concern the 
coupled Nanowire Drift-Diffusion Poisson NDDP system
\eqref{driduf}--\eqref{bdy3} and are divided into two parts.
In a first part we study the derivation of the model \eqref{driduf}--\eqref{density} in a wire with infinite extension
from a kinetic level description. This latter model describes 
the interaction of charged particles with phonons at thermal equilibrium.
Assuming that the potential $V$ is given, we are able to state
in Theorem \ref{th_diffusive_limit} the convergence of this model 
towards the NDDP model.

In a second part, we focus on the study of the NDDP model in the bounded device $\Omega=[0,L]\times\omega_z$.
We will make the following assumptions~:
\begin{assumption}\label{assumption_D}
The function $D$ is assumed to be a $\mathcal{C}^1$ function on $\overline{\Omega}$ and there exists two nonnegative constants $D_1$ and $D_2$ such that $0<D_1 \leq D\leq D_2$.
\end{assumption}
\begin{assumption}\label{assumption_init_N}
The initial condition satisfies $N_s^0 \log N_s^0 \in L^1(\omega_x)$ and $N_s^0 \geq 0$ a.e. And we denote $\mathcal{N}_I = \int_{\omega_x} N_s^0 dx$.
\end{assumption}
\begin{assumption}\label{assumption_bord}
The boundary data for the surface density $N_b$ is a positive constant. The Dirichlet boundary condition for the potential satisfies $V_b\in C^2(\partial \omega_x\times \omega_z)$ and the compatibility condition $\partial_z V_b (z)=0$ for all $z\in\partial \omega_z$.
\end{assumption}

\noindent The main result is the following existence theorem :
\begin{theorem}\label{main_theorem}
Let $T>0$. Under Assumptions \ref{assumption_D}, \ref{assumption_init_N} 
and \ref{assumption_bord}, 
the Nanowire Drift-Diffusion Poisson system \eqref{driduf}--\eqref{bdy3} 
admits a weak solution such that
\begin{equation*}
N_s\log N_s\in L^{\infty}([0,T];L^1(\omega_x))\hspace{0.5cm}\mbox{ and }
\hspace{0.5cm}\sqrt{N_s}\in L^2([0,T];H^1(\omega_x)),
\end{equation*}
\begin{equation*}
V\in L^{\infty}([0,T];H^1(\Omega)).
\end{equation*}
\end{theorem}

To prove this result, we follow the idea proposed in \cite{ddsplog} which
relies strongly on the estimate on the relative entropy that we will prove
in Section \ref{apriori}.
The main difficulty is due to the quantum confinement for which 
we need some sharp estimates on the quantities provided by the Bloch problem. 
Such estimates are given in Section \ref{spectrum}.
Then a priori estimates give a functional framework for the quantities 
$N_s$ and $V$. Since $N_s$ will be defined only in $L\log L$, we need 
to regularized the system. For the regularized system we obtain existence
of solutions such as in \cite{ddsplog} and we recover a solution of the 
non-regularized system by passing to the limit in the regularization.

This paper is organized as follows. In Section \ref{derivation}, 
we describe the derivation of the model from a kinetic model
taking into account the interactions of the charged particles 
with phonons.
Section \ref{analysis} is devoted to the proof of Theorem 
\ref{main_theorem}. We first state estimates on the eigen-elements 
defining the Nanowire quantities. Then, we define the regularization 
of the model. We prove a priori estimates for this regularized system.
Next, the regularized Nanowire Poisson system is analyzed. Finally,
we prove Theorem \ref{main_theorem} by passing to the limit in 
the regularization.

\section{Diffusive limit}
\label{derivation}

\subsection{Kinetic description}\label{kinetic_description}
The drift-diffusion model can be derived from kinetic theory 
when the mean free path related to particle interactions 
with a thermal bath is small compared to the system length-scale
\cite{nbadegond,poupaud}.
In this section, we present the derivation of this model from
the Boltzmann equation describing collisions of charged particles
with phonons at thermal equilibrium. This equation governs the evolution of the distribution function 
$f_n(t,x,p)$ on the $n^{th}$ band whose energy is given by $E_n+V_{nn}$.
Here and in the following, we shall use the notation $f_n$ 
for a function depending on the $n^{th}$ band, and the notation 
$f=(f_n)_{n\geq1}$ when the entire set of bands is taken into account. 
The time variable $t$ is nonnegative, the position variable is denoted 
$x$ and the momentum variable $p$.
The equation writes \cite{poupaud,diff1D}
\begin{equation}\label{boltzman}
\partial_t f_{n}^{\eta} +\frac{1}{\eta} \Big(v_n \partial_xf_{n}^{\eta} -\partial_x V_{nn} \partial_p f_{n}^{\eta}\Big)=\frac{1}{\eta^2}\mathcal{Q}_B(f^{\eta})_n,
\end{equation}
where $\eta$ is the scaled mean free path, assumed to be small.
In this equation, $v_n$ is the velocity given by 
$v_n(p)=\frac{p}{m^*_n}$, $m^*_n$ the $n^{th}$ band effective mass (\ref{effective_mass}) and $V_{nn}(t,x)$ the effective potential 
energy associated to the $n^{th}$ band (\ref{v_nn}).
This equation is completed by the initial data denoted $f^0$.

The collision operator $\mathcal{Q}_B$, describing the scattering between 
electrons and phonons, is assumed in the linear BGK approximation 
for Boltzmann statistics. It reads
\begin{equation}\label{op_collision}
\mathcal{Q}_B(f)_n=\sum_{n'=1}^{+\infty}\int_{\mathbb{R}} \alpha_{n,n'}(p,p')\Big(\mathcal{M}_n(p)f_{n'}(p')-\mathcal{M}_{n'}(p')f_n(p)\Big)dp'
\end{equation}
where the function $\mathcal{M}_n$ is the Maxwellian
\begin{equation}
\mathcal{M}_n(t,x,p)=\frac{1}{\sqrt{2\pi m^*_n}\mathcal{Z}(t,x)}e^{-\big(\frac{p^2}{2m^*_n}+E_n+V_{nn}(t,x)\big)}
\end{equation}
normalized such that
\begin{equation}\label{norm}
\sum_{n=1}^{+\infty}\int_{\mathbb{R}}\mathcal{M}_ndp=1.
\end{equation}
The repartition function $\mathcal{Z}$ is thus given by
\begin{equation}\label{def_Z}
\calZ(t,x)=\sum_{n=1}^{+\infty}e^{-\big(E_n+V_{nn}(t,x)\big)}.
\end{equation}
The energies $E_n$ correspond to the eigenvalues of the problem 
(\ref{eigen_value_pb}). We notice that Assumption \ref{hyp_WL} 
allows us to give a sense to this definition of $\mathcal{Z}$ since 
$\calZ\leq \sum_n e^{-E_n} \leq \sum_n e^{-\Lambda_n}<\infty$, where
$\Lambda_n$ are the eigenvalues of the Laplacian operator (see Section \ref{spectrum}).
\begin{assumption}\label{assumption_alpha}
The cross section $\alpha$ is symmetric and bounded from above and below :
\begin{equation*}
\exists \alpha_1,\alpha_2 >0,\hbox{  }0<\alpha_1\leq \alpha_{n,n'}(p,p') \leq \alpha_2,\hbox{  } \forall n,n'\geq1,\hbox{  } \forall p\in\mathbb{R},\hbox{  } \forall p'\in\mathbb{R}.
\end{equation*}
\end{assumption}

In the diffusion approximation, boundary layers appears at the frontier 
of the domain.
Since the study of this phenomena is far from the scope of this
paper, we consider the limit in the case where the spatial domain is $\RR$, 
assuming that there is no charged carriers at infinity such that
$\lim_{x\to\pm \infty} f^\eta_n(t,x,p)=0$.
For the rigorous analysis of boundary layers in the diffusion 
approximation, we refer the reader to \cite{poupaud,masmoudi}. 


Let us recall an existence result for our problem. 
It is a direct corollary of well known existence results on the Boltzmann 
equation (see e.g. \cite{tayeb,poupaud} and references therein).
\begin{theorem}
Let us assume that the potential $V\geq 0$ is given in $L^{\infty}([0,T];H^2(\mathbb{R}\times\omega_z))$
and that the initial data satisfies 
$f^{0} \in l^1(L^{\infty}(\RR\times\mathbb{R}))$ and $f^{0}\geq0$.
For fixed $\eta > 0$, under Assumption \ref{assumption_alpha}, 
the problem (\ref{boltzman})-(\ref{def_Z}) admits a unique weak solution $f^{\eta} \in L_{loc}^{\infty}(\mathbb{R}^+, l^1(L^1(\RR\times \mathbb{R})))$ and $f^{\eta} \geq 0$.
\end{theorem}

\subsection{Properties of the collision operator}

We present some well known properties of the collision operator $\mathcal{Q}_B$ 
defined by (\ref{op_collision}). 
In this section, the time variable $t$ and the position variable $x$ are considered as parameters, thus we omit to write the dependence on $t$ and $x$. 
We define the weighted space
\begin{equation}
L^2_{\mathcal{M}}= \{ f=(f_n)_{n\geq1} \hbox{ such that } \sum_{n=1}^{+\infty} \int_{\mathbb{R}} \frac{f_n^2}{\mathcal{M}_n} dp < +\infty \},
\end{equation}
which is a Hilbert space with the scalar product
\begin{equation}
\langle f,g\rangle _{\mathcal{M}} = 
\sum_{n=1}^{+\infty} \int_{\mathbb{R}} \frac{f_n g_n}{\mathcal{M}_n} dp.
\end{equation}
We have the following properties for $\mathcal{Q}_B$ whose 
proofs can be found in \cite{ddspnum} (section 3.1).
\begin{proposition}\label{prop_Q}
We assume that the cross section $\alpha$ satisfies Assumption \ref{assumption_alpha}. Then the following properties hold for $\mathcal{Q}_B$ :\vspace{0.3cm}\\
(i) $\sum_{n\geq1} \int_{\mathbb{R}} \mathcal{Q}_B(f)_{n} dp =0$.\vspace{0.3cm}\\
(ii) $\mathcal{Q}_B$ is a linear, selfadjoint and negative bounded operator on $L^2_{\mathcal{M}}$.\vspace{0.3cm}\\
(iii) $Ker \ \mathcal{Q}_B = \{ f\in L^2_{\mathcal{M}},\hbox{  such that }\exists N_s\in \mathbb{R}, f_n=N_s\mathcal{M}_n \}$ 
and $(Ker \ \mathcal{Q}_B)^\bot = Im \ \mathcal{Q}_B$.\vspace{0.3cm}\\
(iv) If $\mathcal{P}$ is the orthogonal projection on $Ker \ \mathcal{Q}_B$ 
with the scalar product $\langle.,.\rangle_{\mathcal{M}}$, then
 $$
 -\langle\mathcal{Q}_B(f),f\rangle_{\mathcal{M}} \geq \alpha_1 
\|f-\mathcal{P}(f)\|^2_{\mathcal{M}}.
$$
\end{proposition}

The third point of Proposition \ref{prop_Q} implies that the equation
$\mathcal{Q}_B(f)=h$ admits a solution in $L^2_{\mathcal{M}}$ iff 
$h\in (Ker \ \mathcal{Q}_B)^{\perp}$. 
Moreover, this solution is unique if we impose $f\in(Ker \ \mathcal{Q}_B)^{\perp}$ where
$(Ker \ \mathcal{Q}_B)^{\perp}=\{f \hbox{  such that  } \sum_{n=1}^{+\infty}\int_{\mathbb{R}} f_n dp =0 \}$.
As a consequence, we can define~:
\begin{definition}\label{consequence_Q}
There exists $\Theta \in L^2_{\mathcal{M}}$ such that for all $n\geq1$,
\begin{equation}\label{def_theta}
\mathcal{Q}_B(\Theta)_n = -\frac{p}{m^*_n}\mathcal{M}_n \hspace{0.5cm}\mbox{ and }\hspace{0.5cm} \sum_{n=1}^{+\infty} \int_{\mathbb{R}} \Theta_n dp =0.
\end{equation}
We define the nonnegative diffusion coefficient by
\begin{equation}\label{definition_D}
D=\sum_{n=1}^{+\infty} \int_{\mathbb{R}} \frac{p}{m^*_n} \Theta_n  dp.
\end{equation}
\end{definition}

\begin{remark} {\sc Particular case when $\alpha$ is constant.}\\
Let us assume that for all $n,n',k,k'$, $\alpha(n,n',k,k')=1/\tau$, where $\tau$ is a relaxation time. After calculations,
\begin{equation*}
\Theta_n(p)=\tau \frac{p}{m^*_n} \mathcal{M}_n(p) \hspace{1cm}\hbox{and}\hspace{1cm}
D=\tau\sum_{n=1}^{+\infty} \frac{e^{-(E_n+V_{nn})}}{m_n^* \sum_{m=1}^{+\infty} e^{-(E_m+V_{mm})}}.
\end{equation*}
We emphasize that this expression is slightly different to the formula used in \cite{ddspnum} for example where $D$ is defined as $\tau/m$.
\end{remark}

\subsection{Asymptotic expansion for the diffusive limit}

Let us consider a solution $f_n^\eta$ of the Boltzmann equation
(\ref{boltzman}) and assume that it admits a Hilbert expansion
$$
f_n^{\eta}=f_{0,n}+\eta f_{1,n}+\eta^2f_{2,n}+...
$$
Inserting this decomposition in (\ref{boltzman}) and identifying with respect to powers of $\eta$, we obtain
\begin{equation}\label{1}
\mathcal{Q}_B(f_0)_n=0.
\end{equation}
\begin{equation}\label{2}
v_n\partial_x f_{0,n}-\partial_x V_{nn} \partial_p f_{0,n}=\mathcal{Q}_B(f_1)_n.
\end{equation}
\begin{equation}\label{3}
\partial_t f_{0,n} +v_n\partial_xf_{1,n}-\partial_x V_{nn} \partial_p f_{1,n}=\mathcal{Q}_B(f_2)_n.
\end{equation}

With (\ref{1}), we get $f_{0}\in Ker \ \mathcal{Q}_B$. Thus, (iii) of Proposition \ref{prop_Q} gives 
\begin{equation}\label{expression_f0}
f_{0,n}=N_s\mathcal{M}_n.
\end{equation}
Injecting this expression in (\ref{2}) it follows, after calculations
$$
\mathcal{Q}_B(f_1)_n=H_n:=v_n\mathcal{M}_n\Big(\partial_xN_s+N_s\partial_xV_s\Big),
$$
where $V_s=-\ln\mathcal{Z}$.
By Proposition \ref{prop_Q}, $f_{1}$ exists iff $H \in (Ker \ \mathcal{Q}_B)^\perp$. Because we have $\int v_n\mathcal{M}_n dp=0$ ($v_n\mathcal{M}_n$ is an odd function), this condition is true. We choose $\Theta$ as proposed in Definition \ref{consequence_Q}. Thus,
\begin{equation}\label{expression_f1}
f_{1,n}=-\Theta_n\Big(\partial_xN_s+N_s\partial_xV_s\Big).
\end{equation}

By Proposition \ref{prop_Q}, (\ref{3}) has a solution iff, 
\begin{equation*}
\sum_{n=1}^{+\infty}\int_{\mathbb{R}}\Big(\partial_t f_{0,n} +v_n\partial_xf_{1,n}-\partial_x V_{nn} \partial_p f_{1,n}\Big)dp=0.
\end{equation*}
Using \eqref{expression_f0}, \eqref{expression_f1} and (\ref{definition_D}) 
we have formally obtained the drift-diffusion equation \eqref{driduf}.

\subsection{A convergence proof of the derivation}

In this Section, we investigate the rigorous diffusive limit 
of the Boltzmann equation (\ref{boltzman}) as $\eta \rightarrow 0$.
This study is proposed in the simplified case where the potential
$V$ is given and regular.
The diffusive limit of a coupled Boltzmann transport equation with
the Poisson equation is studied in \cite{masmoudi} and with quantum 
confinement in \cite{diff1D}.
The main result is the following theorem~:

\begin{theorem}\label{th_diffusive_limit}
Let us assume that the potential $V\geq 0$ is given in 
$L^{\infty}([0,T];H^2(\mathbb{R}\times\omega_z))$ and $\partial_tV$ is bounded in 
$L^\infty([0,T]\times\mathbb{R}\times\omega_z)$
and that the initial data satisfies
$f^{0} \in l^1(L^{\infty}(\mathbb{R}^2))$ and $f^{0}\geq0$. 
Moreover, let $T>0$ and let $(f_n^{\eta})_{n\geq1}$ be a solution 
of the Boltzmann equation (\ref{boltzman})-(\ref{def_Z}). 
Then, under Assumption \ref{assumption_alpha}, $N_s^{\eta}$ 
defined by $N_s^{\eta} := \sum_{n\geq1} \int_{\mathbb{R}} f_n^{\eta} dp$ converges weakly towards $N_s \in L^2([0,T]\times\RR)$
solution of the drift-diffusion equation :
\begin{equation*}
\partial_t N_s - \partial_x\Big( D(\partial_x N_s + N_s \partial_x V_s) \Big) =0,
\end{equation*}
where $V_s(t,x)=-\ln\Big(\sum_n e^{-\big(E_n+V_{nn}(t,x)\big)}\Big)$, with the initial data
$N_s^{0}(x)=\sum_n \int_{\mathbb{R}} f_n^{0}(x,p)\,dp$.
\end{theorem}

\begin{notation}
For the proof, we introduce the function $M_n=\frac{1}{\sqrt{2\pi m_n^*}} e^{-\big(\frac{p^2}{2m^*_n} +E_n+V_{nn}\big)}$ which is such that 
\begin{equation*}
v_n \partial_x M_n - \partial_x V_{nn} \partial_p M_n=0.
\end{equation*}
Moreover, for $T>0$, we consider the Banach spaces $X=L^{\infty}([0,T];L^2_{M(t)})$ and $Y=L^2([0,T];L^2_{M^{-1}(t)})$.
\end{notation}

\begin{lemma}\label{lemma1diff}
Assume $f^{0} \in l^1(L^1_{x,p})\cap L^2_{\mathcal{M}(t=0)}$ and $V$ is given such as in Theorem \ref{th_diffusive_limit}. Then, the unique solution $f^{\eta}$ of the Boltzmann equations (\ref{boltzman}) in $L^{\infty}([0,T];l^1(L^1_{x,p}))$ is in $X$. Moreover, $f^{\eta}$ is bounded in $X$ independently of $\eta$.
\end{lemma}

\begin{proof}
Assuming that all the functions are regular enough, we multiply (\ref{boltzman}) by $f_n^{\eta}/M_n$ and integrate. We obtain
\begin{equation}\label{eq_bolt_int}
\frac{d}{dt} \sum_{n=1}^{+\infty} \intdouble \frac{(f_n^{\eta})^2}{2M_n}\,dxdp 
- \sum_{n=1}^{+\infty} \intdouble \partial_t V_{nn} \frac{(f^{\eta}_n)^2}{2M_n}\,dxdp
=\frac{1}{\eta^2} \sum_{n=1}^{+\infty} \intdouble \mathcal{Q}_B(f^{\eta})_n \frac{f^{\eta}_n}{M_n} \,dxdp.
\end{equation}
By assumption there exists $\mu\geq0$ such that $|\partial_t V_{nn}| \leq \mu$ on $[0,T]\times\RR$.
We define
\begin{equation}
X^{\eta}(t) =\sum_{n=1}^{+\infty} \intdouble \frac{(f_n^{\eta})^2}{2M_n}\,dxdp 
\hspace{0.5cm}\mbox{and}\hspace{0.5cm} S^{\eta}(t)=-\frac{1}{\eta^2} \sum_{n=1}^{+\infty} \intdouble \mathcal{Q}_B(f^{\eta})_n \frac{f^{\eta}_n}{M_n} dx dp.
\end{equation}
Since $\mathcal{Q}_B$ is negative, $S^{\eta}(t)\geq 0 \ \forall t\in[0,T]$. So, (\ref{eq_bolt_int}) gives
\begin{equation}\label{eq_Xeta}
\frac{d X^{\eta}}{dt} - \mu X^{\eta} \leq - S^{\eta}.
\end{equation}
Integrating this inequality allows to conclude the proof.

To justify all calculations, we regularized the problem and consider $f_{R}$ a solution of the regularized truncated problem, $f_{R}\in \mathcal{D}([0,T]\times\mathbb{R}^2)$. Thus $f_R$ satisfies (\ref{eq_bolt_int}) and
\begin{equation*}
\frac{d}{dt} \sum_{n=1}^{+\infty} \intdouble \frac{f^2_{R}}{2M_n}\,dxdp \leq \mu \sum_{n=1}^{+\infty} \intdouble \frac{f^2_{R}}{2M_n}\,dxdp.
\end{equation*}
Thus $f_R$ is bounded in $X$ independently of $R$. We can extract a subsequence converging towards a function $g\in X$ in $X$-weak$^*$. We know moreover that $f_R$ satisfies the Cauchy criterion in $L^{\infty}([0,T],l^1(L^1_{x,p}))$ as a solution of the truncated problem. Thus, $f_R$ converges strongly towards $f$ in this space. By uniqueness of the weak$^*$ limit, $g=f$ a.e.
\end{proof}

\begin{lemma}\label{lemma2diff}
Under the assumption of Lemma \ref{lemma1diff}, 
there exists $f \in X$ and $N_s\in L^2([0,T]\times\RR)$ such that, 
up to an extraction~:\\
(i) $f^{\eta} \rightharpoonup f$ in $X$-weak$^*$.\\
(ii) $N_s^{\eta} \rightharpoonup N_s$ in $L^2_{t,x}$-weak.\\
(iii) If we define the current by
\begin{equation}\label{def_current_eta}
J^{\eta}=\frac{1}{\eta} \sum_{n=1}^{+\infty} \int_{\mathbb{R}} v_n f_n^{\eta} dp,
\end{equation}
then $J^{\eta} \rightharpoonup J$ in $L^2_{t,x}$-weak. 
Moreover, $f_n=N_s\mathcal{M}_n$ a.e. for all $n$.
\end{lemma}

\begin{proof}
We integrate (\ref{eq_Xeta}) between $0$ and $t$, for all $t\in[0,T]$. It gives
\begin{equation*}
X^{\eta}(t) -\mu \int_{0}^{t} X^{\eta}(s) ds + \int_{0}^{t} S^{\eta}(s) ds \leq X^{\eta}(0).
\end{equation*}
Thus, there exists a constant $C>0$ such that
\begin{equation}
0\leq \int_{0}^{T}S^{\eta}(s) ds \leq C.
\end{equation}
If $\mathcal{P}$ is the orthogonal projection on $Ker \ \mathcal{Q}_B$ defined as in Proposition \ref{prop_Q}, 
there exists $N^{\eta}(t,x)$ such that $\mathcal{P}(f^{\eta})_n=N^{\eta} \mathcal{M}_n$. Moreover, $f^{\eta} - \mathcal{P}(f^{\eta}) \in (Ker \ \mathcal{Q}_B)^{\perp}$ that is to say that $\sum_{n\geq1} \int \big( f_n^{\eta}- \mathcal{P}(f^{\eta})_n \big)dp=0$. We conclude, using (\ref{norm}), that
\begin{equation}
N^{\eta}(t,x) =\sum_{n=1}^{+\infty} \int_{\mathbb{R}} f_n^{\eta}(t,x,p) dp := N_s^{\eta}(t,x).
\end{equation}
We can easily show that (iv) of Proposition \ref{prop_Q} is also true for the scalar product $\langle.,.\rangle_M$. 
We obtain the bound
\begin{equation}\label{bound_feta}
\sum_{n=1}^{+\infty} \int_0^T \intdouble \frac{(f^{\eta}-N_s^{\eta}\mathcal{M}_n)^2}{M_n} \,dxdpdt \leq \frac{\eta^2}{\alpha_1}\int_0^T S^{\eta}(s)\,ds \leq C \eta^2. 
\end{equation}
We verify that $\|f^{\eta}\|_X$ is bounded. Thus, we can extract a subsequence satisfying (i).

Then, from (\ref{bound_feta}), we have
\begin{equation*}
\sum_{n=1}^{+\infty} \int_0^T \intdouble \frac{(N_s^{\eta}\mathcal{M}_n)^2}{M_n}\,dxdpdt = \sum_{n=1}^{+\infty} \int_0^T \intdouble \Big(\frac{N_s^{\eta}}{\mathcal{Z}}\Big)^2 M_n \,dxdpdt\leq C.
\end{equation*}
It provides that $\|N_s^{\eta}/\mathcal{Z}\|_Y$ is bounded. 
Thus there exists $\rho\in Y$ such that, up to an extraction, for all $\phi\in L^2_tL^2_{M^{-1}}$,
\begin{equation}\label{bound_feta2}
\sum_{n=1}^{+\infty} \int_0^T \intdouble\frac{N_s^{\eta}}{\mathcal{Z}}M_n \phi_n\,dxdpdt 
\rightarrow \sum_{n=1}^{+\infty} \int_0^T \intdouble\rho M_n \phi_n\,dxdpdt.
\end{equation}
For $\xi \in L^2([0,T]\times\RR)$ take $\phi_n (t,x,p)=\xi(t,x)$ for all $n\geq1$, we easily verify that $\phi \in L^2_tL^2_{M^{-1}}$. So, if we define $N_s := \rho \mathcal{Z} \in L^2([0,T]\times\RR)$, we find
\begin{equation*}
\int_0^T \int N_s^{\eta} \xi \Bigg(\sum_{n=1}^{+\infty} \int \frac{M_n}{\mathcal{Z}} dp \Bigg) dx dt \rightarrow \int_0^T \int \rho \mathcal{Z} \xi \Bigg(\sum_{n=1}^{+\infty} \int\frac{M_n}{\mathcal{Z}} dp \Bigg) dx dt.
\end{equation*}
This proves (ii). Moreover, from (\ref{bound_feta2}), $N^{\eta}_s\frac{M_n}{\mathcal{Z}}=N_s^{\eta} \mathcal{M}_n \rightharpoonup \rho M_n$ in $X$-weak$^{*}$. From (\ref{bound_feta}), $f^{\eta}$ and $N_s^{\eta} \mathcal{M}_n$ have the same weak limit. So, for all $n\geq1$,
\begin{equation}\label{f_0}
f=\rho M_n = N_s \mathcal{M}_n \ \ \hbox{a.e.}
\end{equation}

Finally, we have
\begin{equation*}
J^{\eta}=\frac{1}{\eta} \sum_{n=1}^{+\infty} \int_{\mathbb{R}} v_n f_n^{\eta} dp = \frac{1}{\eta} \sum_{n=1}^{+\infty} \int_{\mathbb{R}} v_n \big(f_n^{\eta} -N_s \mathcal{M}_n \big) dp.
\end{equation*}
Applying the Cauchy-Schwarz inequality
\begin{equation}
J^{\eta} \leq \Bigg(\sum_{n=1}^{+\infty}\int_{\mathbb{R}} v_n^2 M_n dp \Bigg)^{1/2} \Bigg(\sum_{n=1}^{+\infty}\int_{\mathbb{R}} \frac{(f_n^{\eta} -N_s \mathcal{M}_n)^2}{\eta^2 M_n} dp \Bigg)^{1/2}.
\end{equation}
With (\ref{bound_feta}), we deduce that $J^{\eta}$ is bounded in $L^2_{t,x}$.
\end{proof}

{\bf Proof of Theorem \ref{th_diffusive_limit}.}
Integrating the Boltzmann equation \eqref{boltzman} with respect to $p$,
we find the conservation law
\begin{equation}\label{conservation_law}
\partial_t N_s^{\eta} + \partial_x J^{\eta} =0
\end{equation}
where $J^{\eta}$ is the current defined in (\ref{def_current_eta}). Considering the function $\Theta$ defined in (\ref{def_theta}), we have
\begin{equation*}
J^{\eta} = \frac{1}{\eta} \sum_{n=1}^{+\infty} \int_{\mathbb{R}} v_n f_n^{\eta} dp 
=\frac{1}{\eta}\langle v_n\mathcal{M}_n,f_n^{\eta}\rangle_{\mathcal{M}}
=-\frac{1}{\eta}\langle\mathcal{Q}_B(\Theta)_n,f_n^{\eta}\rangle_{\mathcal{M}}.
\end{equation*}
The selfadjointness of $\mathcal{Q}_B$ gives
\begin{equation}\label{j_eta}
-J^{\eta}=\frac{1}{\eta}\langle\Theta_n,\mathcal{Q}_B(f^{\eta})_n\rangle_{\mathcal{M}}
=\sum_{n=1}^{+\infty} \int_{\mathbb{R}} \frac{\Theta_n \mathcal{Q}_B(f^{\eta})_n}{\eta \mathcal{M}_n} dp.
\end{equation}

Now, to establish the rigorous limit $\eta\rightarrow0$, we use the weak formulation of (\ref{boltzman}) and (\ref{conservation_law})~: for all $\psi\in\mathcal{C}^1([0,T]\times\RR^2)$ compactly supported and for all $\phi\in\mathcal{C}^1([0,T]\times\RR)$ compactly supported :
\begin{equation}\label{varia_B}
\inttriple f_n^{\eta}\big(-\eta \partial_t\psi -v_n\partial_x \psi +\partial_x V_{nn} \partial_p \psi\big) dxdpdt-\frac{1}{\eta}\inttriple \mathcal{Q}_B(f^{\eta})_n \psi dxdpdt=0,
\end{equation}
and
\begin{equation}\label{varia_CE}
-\intdouble N_s^{\eta} \partial_t \phi dxdt-\intdouble J^{\eta} \partial_x \phi dxdt=0.
\end{equation}
At this point we use Lemmas of Appendix A. On the one hand, from Lemma \ref{lemme1_appendix1}, we can substitute $\psi$ by $\phi$ in (\ref{varia_B}). Summing with respect to $n$, we immediately obtain (\ref{varia_CE}). On the other hand, Lemmas \ref{lemme1_appendix1} and \ref{lemme2_appendix1} prove that we can choose $\psi=\phi \frac{\Theta_n}{\mathcal{M}_n}$ in (\ref{varia_B}) for all $\phi \in \mathcal{C}^1([0,T]\times\RR)$ compactly supported. Summing with respect to $n$ and using (\ref{j_eta}), we find after calculations
\begin{eqnarray}
&& -\eta \sum_{n=1}^{+\infty}\inttriple f_n^{\eta} \partial_t\Big(\phi \frac{\Theta_n}{\mathcal{M}_n}\Big) dxdpdt 
- \sum_{n=1}^{+\infty}\inttriple v_n \frac{f_n^{\eta} }{\mathcal{M}_n} \partial_x \big(\Theta_n \phi \big)dxdpdt\nonumber\\
&& +\sum_{n=1}^{+\infty}\inttriple \phi \frac{f_n^{\eta}}{\mathcal{M}_n} 
\big(\partial_x V_{nn} \partial_p \Theta_n -v_n \Theta_n \partial_x \ln\mathcal{Z} \big) dxdpdt + \intdouble J^{\eta} \phi dxdt=0.
\end{eqnarray}
Using Lemma \ref{lemma2diff}, we have the weak convergence of $f_n^\eta$,
$N_s^\eta$ and $J^\eta$. Moreover,
we have that $f^{\eta}_n \in L^2_{\mathcal{M}}$ thus the limit of the first term vanishes thanks to Lemma \ref{lemme2_appendix1}. Moreover, this same Lemma proves that $v_n\frac{\Theta_n}{\mathcal{M}_n}$ and $v_n\frac{\partial_s \Theta_n}{\mathcal{M}_n}$ (for $s=t$, $x$ and $p$) are in $X$. 
Finally, since by assumption $V_{nn}$ is bounded in $L^{\infty}([0,T];H^2(\RR))$,
we can pass to the limit and we obtain
\begin{equation*}
\intdouble J \phi\, dxdt= \sum_{n=1}^{+\infty}\inttriple \Bigg( v_n N_s \partial_x \big(\Theta_n\phi\big) 
-N_s\phi\big(\partial_x V_{nn} \partial_p \Theta_n 
-v_n \Theta_n \partial_x \ln\mathcal{Z} \big)\Bigg) dxdpdt,
\end{equation*}
which is the weak formulation of 
$J=-D(\partial_x N_s + N_s \partial_x V_s)$. 
\qed

\section{Analysis of the Nanowire Drift-Diffusion-Poisson system}
\label{analysis}

\subsection{Spectral properties}
\label{spectrum}

In this Section, we investigate some technical Lemmas concerning 
spectral properties of the Hamiltonian defined in \eqref{eigen_value_pb}.
As in Section \ref{kinetic_description} we denote by $\Lambda_n$ the eigenvalues of the Laplacian, i.e.
$$\left\{
\begin{array}{l}
\dsp -\Delta u_n = \Lambda_n u_n , \qquad \mbox{ on }\quad 
\calU=(-1/2,1/2)\times\omega_z, \\[2mm]
\dsp u_n(-1/2,z')=u_n(1/2,z'), \qquad 
u_n(y,z') = 0 \quad \mbox{ on }(-1/2,1/2)\times\partial\omega_z.
\end{array}\right.
$$
From the min-max principle it is clear that for nonnegative
potential $W_\calL$, we have $E_n\geq \Lambda_n$. Moreover,
the eigenfunctions $u_n$ satisfy $u_n=u_n^1u_n^2$ where
$(u_n^1)_{n\in \NN}$ are eigenvectors of the Laplacian in the $y$-direction
with periodic boundary conditions and
$(u_n^2)_{n\in \NN}$ are eigenvectors of the Laplacian in the $z$-direction 
with Dirichlet boundary conditions.
From well-known properties of eigenvalues of the Laplacian-Dirichlet \cite{henrot}, 
we deduce that for all $\lambda>0$, $\sum_{n\geq1} e^{-\lambda\Lambda_n}<+\infty$. Thus
\beq\label{sumEn}
\forall\, \lambda >0, \quad \sum_{n=1}^{+\infty} e^{-\lambda E_n} <\infty.
\eeq

In the following we will make use of the notation~:
$L^p_zL^q_x(\Omega)=\{u\in L^1_{loc} (\Omega) \mbox{ s.t. }
\|u\|_{L^p_zL^q_x(\Omega)} =(\int_{\omega_z}
\|u(\cdot,z)\|_{L^q(\omega_x)}^p\,dz)^{1/p}<+\infty \}$.
We recall that we have the Sobolev embedding 
$H^1(\Omega)\hookrightarrow L^2_zL^\infty_x(\Omega)$ where 
$\Omega=\omega_x\times\omega_z\subset \RR\times\RR^2$ (see \cite{vsp}).

\begin{lemma}\label{lemme_spec}
Under Assumption \ref{hyp_WL} we have that for all $n\in \NN$
\beq\label{born_chi}
\|\chi_n\|_{L^\infty(\Omega)} \leq C(E_n + \|W_{\calL}\|_{L^\infty}),
\eeq
where $C$ stands for a nonnegative constant.
Therefore we have
\beq\label{estimate_g_nn}
\|g_{nn}\|_{L^{\infty}(\omega_z)} \leq C(E_n+ \|W_{\calL}\|_{L^\infty})^2,
\eeq
\beq\label{estimate_v_nn}
\|V_{nn}\|_{L^\infty(\omega_x)} \leq C\|V\|_{H^1(\Omega)}
(E_n+ \|W_{\calL}\|_{L^\infty})^2.
\eeq
\end{lemma}
\begin{proof}
We notice first that since $W_{\calL}$ (and $\chi_n$) is 1-periodic in $x$, 
we have that $\|W_{\calL}\|_{L^\infty(\calU)}=\|W_{\calL}\|_{L^\infty(\Omega)}$.
From \eqref{eigen_value_pb}, we deduce that $\Delta \chi_n\in L^2(\Omega)$ and 
$$
\|\Delta\chi_n\|_{L^2}\leq 2 (E_n +  \|W_{\calL}\|_{L^\infty}).
$$
Using the elliptic regularity for the Laplacian operator, we deduce 
\eqref{born_chi} thanks to the Sobolev embedding 
$H^2(\Omega)\hookrightarrow L^\infty(\Omega)$. The estimates 
\eqref{estimate_g_nn} and \eqref{estimate_v_nn} follows directly
from a Cauchy-Schwarz inequality and the Sobolev embedding 
$H^1(\Omega)\hookrightarrow L^2_zL^\infty_x(\Omega)$.
\end{proof}

\begin{lemma}\label{proposition_S}
{\bf Properties of $\mathcal{S}[V]$}. Let us assume that $V \in H^1(\Omega)$ with $V\geq 0$. Then the 
function $\mathcal{S}[V]$, defined in (\ref{def_S}), satisfies~:\\
(i)  $\mathcal{S}[V]$ is bounded i.e. $|\mathcal{S}[V]| < +\infty$.  \\
(ii) If moreover $\widetilde{V} \in H^1(\Omega)$ with $\widetilde{V}\geq 0$, 
then there exists a nonnegative constant $C$ such that~:
\begin{equation}\label{bornS}
\| \mathcal{S}[V]-\mathcal{S}[\widetilde{V}] \|_{L^2(\Omega)} \leq C \| V-\widetilde{V} \|_{L^2(\Omega)}.
\end{equation}
\end{lemma}

\begin{proof}
{\bf (i)} First, we study the coefficient $\mathcal{Z}$. We have
\begin{equation*}
\mathcal{Z}(x)= \sum_{n=1}^{+\infty} e^{-\big(E_n +V_{nn}(x)\big)} \geq e^{-\big(E_1 +\|V_{11}\|_{L^{\infty}(\omega_x)}\big)}.
\end{equation*}
Then from Lemma \ref{lemme_spec}, when $V \in H^1(\Omega)$, 
there exists a constant $C>0$ such that $\mathcal{Z}(x) > C$.
Using the fact that $V_{nn} \geq 0$ when $V \geq 0$, we get
\begin{equation*}
|\mathcal{S}[V](x,z)|\leq C \sum_{n=1}^{+\infty} e^{-E_n} \| g_{nn}\|_{L^{\infty}(\omega_z)}\leq
C\sum_{n=1}^{+\infty} e^{-E_n}(E_n+\|W_{\calL}\|_{L^\infty})^2,
\end{equation*}
where we use \eqref{estimate_g_nn} for the last inequality.
A direct consequence of (\ref{sumEn}) is that $\sum_{n\geq1} E_n^2 e^{-E_n}$ is finite. Then $\mathcal{S}[V]$ is bounded.

{\bf (ii)} We use the fact that
\begin{equation*}
\mathcal{S}[V] -\mathcal{S}[\widetilde{V}] = 
\int_0^1 \partial_s S[\widetilde{V}+s(V-\widetilde{V})] ds.
\end{equation*}
We define
$\calE(s)=e^{-\big(E_n+\langle\widetilde{V}+s(V-\widetilde{V}),g_{nn}\rangle\big)}$. So,
\begin{equation*}
\partial_s \mathcal{S}[\widetilde{V}+s(V-\widetilde{V})] = 
\frac{\Big(\sum_{n\geq1} \calE'(s)g_{nn}\Big)\Big(\sum_{n\geq1} \calE(s)\Big) - \Big(\sum_{n\geq1} \calE(s) g_{nn} \Big)\Big(\sum_{n\geq1} \calE'(s)\Big)}{\Big(\sum_{n\geq1} \calE(s)\Big)^2}.
\end{equation*}
We have $\calE'(s)=-\langle V-\widetilde{V},g_{nn}\rangle\calE(s)$. 
The first term becomes
\begin{equation*}
\Big| \frac{\sum_{n\geq1} \calE'(s)g_{nn}}{\sum_{n\geq1} \calE(s)} \Big|= \Big| \frac{\sum_{n\geq1} \calE(s)\langle V-\widetilde{V},g_{nn} \rangle g_{nn}}{\sum_{n\geq1} \calE(s)} \Big| \leq C \sum_{n=1}^{+\infty} e^{-E_n} g_{nn} |\langle V-\widetilde{V},g_{nn}\rangle|.
\end{equation*}
Finally,
\begin{equation*}
\Big| \frac{\sum_{n\geq1} \calE'(s)g_{nn}}{\sum_{n\geq1} \calE(s)} 
\Big| \leq C \|V-\widetilde{V} \|_{L^2(\omega_z)} \sum_{n=1}^{+\infty} e^{-E_n} g_{nn} \|g_{nn}\|_{L^2(\omega_z)} \leq C\|V-\widetilde{V} \|_{L^2(\omega_z)},
\end{equation*}
where we use (\ref{estimate_g_nn}) and (\ref{sumEn}) for the last inequality. We can treat the second term in a similar way
\begin{equation*}
\Bigg|\frac{\Big(\sum_{n\geq1} \calE(s) g_{nn} \Big)\Big(\sum_{n\geq1} \calE(s)<V-\widetilde{V},g_{nn}>\Big)}{\Big(\sum_{n\geq1} \calE(s)\Big)^2}\Bigg| 
\leq C \sum_{n=1}^{+\infty} e^{-E_n} |\langle V-\widetilde{V},g_{nn}\rangle|.
\end{equation*}
Consequently, we deduce \eqref{bornS}.
\end{proof}

\subsection{Regularized system}
We define the linear regularization operator by
\begin{eqnarray}
\mathcal{R}^{\epsilon} : L^1(\Omega) &\rightarrow& C^{\infty}(\overline{\Omega}) \nonumber \\
\mathcal{V} & \mapsto & \mathcal{R}^{\epsilon} [\mathcal{V}](x,z) = (\overline{\mathcal{V}} * \xi_{\epsilon,x} * \xi_{\epsilon,z} )|_{\overline{\Omega}},
\end{eqnarray}
where $\overline{\mathcal{V}}$ is the extension of $\mathcal{V}$ by zero outside $\Omega$ and $\xi_{\epsilon,x}$ and $\xi_{\epsilon,z}$ are $C^{\infty}$ nonnegative compactly supported even approximations of the unity, respectively on $\mathbb{R}$ and $\mathbb{R}^2$. We can prove the following properties, using convolution results :

\begin{lemma}\label{lemme_reg}
{\bf Properties of $\mathcal{R}^\epsilon$} :\\
(i) $\mathcal{R}^\epsilon$ is a bounded operator on $L^p_xL^q_z(\Omega)$ for $1 \leq p,q\leq +\infty$ and satisfies for all $\mathcal{V} \in L^p_xL^q_z(\Omega)$,
\begin{equation*}
\|\mathcal{R}^\epsilon[\mathcal{V}]\|_{L^p_xL^q_z(\Omega)} \leq \|\mathcal{V}\|_{L^p_xL^q_z(\Omega)} \hspace{0.5cm}\hbox{and}\hspace{0.5cm}\lim_{\epsilon\rightarrow 0}\| \mathcal{R}^\epsilon[\mathcal{V}]-\mathcal{V}\|_{L^p_xL^q_z(\Omega)}=0.
\end{equation*}

\noindent (ii) $\mathcal{R}^\epsilon$ is selfadjoint on $L^2(\Omega)$.

\noindent (iii) For all $\mathcal{V} \in H^1(\Omega)$,
\begin{equation*}
\nabla_x \mathcal{R}^\epsilon[\mathcal{V}] = \mathcal{R}^\epsilon[\nabla_x \mathcal{V}] \hspace{0.5cm}\hbox{and}\hspace{0.5cm} \lim_{\epsilon \rightarrow 0} \|\nabla_x \mathcal{R}^\epsilon[\mathcal{V}]-\nabla_x \mathcal{V}\|_{L^2(\Omega)} =0.
\end{equation*}
\end{lemma}

\noindent Then the regularized Nanowire Drift-Diffusion Poisson NDDP$_{\epsilon}$ system is defined for $\epsilon \in [0,1]$ by
\begin{equation}\label{driduf_reg}
\partial_t N^\epsilon_s -\partial_x\Big(D(\partial_x N^\epsilon_s + N^\epsilon_s \partial_x V^\epsilon_s) \Big)=0
\end{equation}
and
\begin{equation}\label{pois_reg}
-\Delta V^\epsilon = \mathcal{R}^\epsilon\Big[\frac{N^\epsilon_s}{\mathcal{Z^\epsilon}}\sum_{n=1}^{+\infty}e^{-\big(E_n+V^\epsilon_{nn}\big)} g_{nn}\Big]=\mathcal{R}^\epsilon\Big[N^\epsilon_s \mathcal{S}^{\epsilon}\Big]
\end{equation}
where the regularized quantities are defined by
\begin{equation}\label{v_nn_reg}
V^\epsilon_{nn}(x) =\int_{\omega_z}\mathcal{R}^\epsilon[V^\epsilon(x,z)]g_{nn}(z)dz=<\mathcal{R}^\epsilon[V^\epsilon],g_{nn}>,
\end{equation}
\begin{equation}\label{vs_reg}
V^{\epsilon}_s=-\log \mathcal{Z}^{\epsilon}\hspace{0.5cm}\hbox{with}\hspace{0.5cm} \mathcal{Z}^{\epsilon}[V^{\epsilon}]=\sum_{n=1}^{+\infty}e^{-\big(E_n+V^{\epsilon}_{nn}\big)},
\end{equation}
and
\begin{equation}\label{S_reg}
\mathcal{S}^{\epsilon}[V^{\epsilon}]= \sum_{n\geq1} \frac{e^{-\big(E_n+V^{\epsilon}_{nn}\big)}}{\mathcal{Z}^{\epsilon}} g_{nn}.
\end{equation}
As above, we denote $N_n^\epsilon=u^\epsilon e^{-(E_n+V_{nn}^\epsilon)}$ where
$u^\epsilon$ is the Slotboom variable $u^\epsilon=N_s^\epsilon/\mathcal{Z}^{\epsilon}$.
The initial regularized density $N_s^{\epsilon,0}$ is chosen such that $N_s^{\epsilon,0}=\min(N_s^0,\epsilon^{-1})$. Moreover, the regularized boundary conditions are 
\begin{eqnarray}
\label{bdyreg1}
N^{\epsilon}_s(t,x)=N_b \hspace{1cm} &\hbox{for }& x\in \partial\omega_x,\\ 
\label{bdyreg2}
V^{\epsilon}(t,x,z)=V_b(z) \hspace{1cm} &\hbox{for }& x\in \partial\omega_x,\\
\label{bdyreg3}
\partial_z V^{\epsilon}(t,x,z)=0 \hspace{1cm} &\hbox{for }& z\in \partial\omega_z.
\end{eqnarray}

\begin{remark}
When $\epsilon \rightarrow 0$, $\mathcal{R}^{\epsilon} \rightarrow Id$ and the regularized system \eqref{driduf}--\eqref{bdy3} tends to the unregularized problem \eqref{driduf_reg}--\eqref{bdyreg3}. 
\end{remark}

\subsection{A priori estimates}
\label{apriori}

Let us consider a weak solution $(N_s^\epsilon,V^\epsilon)$ of the regularized problem
(\ref{driduf_reg})--(\ref{S_reg}).
We introduce two extensions $\underline{N_s}$ and $\underline{V}$ of the boundary data. These extensions are respectively defined on $\omega_x$ and $\Omega$ and chosen such that :
\begin{itemize}
\item $\underline{N_s}\in C^2(\omega_x)$, $0<\underline{N_1}\leq \underline{N_s}(x) \leq \underline{N_2}$ where $\underline{N_1}$ and $\underline{N_2}$ are two constants, and $\underline{N_s}|_{\partial\omega_x}=N_b$.
\item $\underline{V}\in C^2(\Omega)$ and satisfies the boundary conditions $\underline{V}|_{\partial \omega_x\times\omega_z}=V_b(z)$ and $\partial_z \underline{V}|_{\omega_x\times\partial\omega_z}=0$.
\end{itemize}
For regular enough domains, these functions exist. From (\ref{v_nn}) with $\underline{V}$ instead of $V^{\epsilon}$, we define $V_{nn}[\underline{V}]$ denoted by $\underline{V_{nn}}$. In the same way, we denote $\underline{\mathcal{Z}}$, $\underline{\mathcal{S}}$, $\underline{N_n}$, $\underline{\rho}$, $\underline{u}$ and $\underline{E_F}$ the quantities associated to $\underline{N_s}$ and $\underline{V}$.

\begin{proposition}\label{prop_entropie1} 
Let $T>0$ and $\epsilon\in[0,1]$. 
Let $(N_s^\epsilon,V^\epsilon)$ be a weak solution of the regularized system NDDP$_{\epsilon}$
(\ref{driduf_reg})--(\ref{bdyreg3}), 
such that $N_s^\epsilon \log N_s^\epsilon \in L^{\infty}([0,T],L^1(\omega_x))$, $V^\epsilon \in L^{\infty}([0,T],H^1(\Omega))$ and $\sqrt{N_s^\epsilon} \in L^2([0,T],H^1(\omega_x))$. 
Then, there exists a nonnegative constant $C_T$ depending only on $T$ 
such that
\begin{equation}\label{estimW}
\forall t \in [0,T],\hspace{1.5cm}  0 \leq W(t) \leq C_T,
\end{equation}
where $W$ is the relative entropy defined by
\begin{equation}\label{def_entropie}
W=\sum_{n=1}^{+\infty} \int_{\omega_x} \Big(N_n^\epsilon \ln\Big(\frac{N_n^\epsilon}{\underline{N_n}}\Big) - N_n^\epsilon + \underline{N_n} \Big)dx + \frac{1}{2}\int_{\Omega} |\nabla (V^\epsilon-\underline{V})|^2dxdz.
\end{equation}
\end{proposition}

\begin{proof}
We remark that
\begin{equation*}
\frac{d}{dt}\sum_{n=1}^{+\infty} \int_{\omega_x} \Big(N_n^\epsilon \ln\Big(\frac{N_n^\epsilon}{\underline{N_n}}\Big) - N_n^\epsilon + \underline{N_n} \Big)dx = \sum_{n=1}^{+\infty} \int_{\omega_x} \partial_t N_n^\epsilon \ln \Big(\frac{N_n^\epsilon}{\underline{N_n}}\Big) dx .
\end{equation*}
By definition, we have $\ln N_n^\epsilon=\ln u^\epsilon -E_n -V_{nn}^\epsilon$.
Using \eqref{driduf_reg}, it leads to
\begin{eqnarray*}
\sum_{n=1}^{+\infty} \int_{\omega_x} \partial_t N_n^\epsilon \ln  \Big(\frac{N_n^\epsilon}{\underline{N_n}}\Big) dx
 &=& \int_{\omega_x} \partial_x \big(De^{-V_s^\epsilon}\partial_xu^\epsilon\big) \ln\Big(\frac{u^\epsilon}{\underline{u}}\Big) dx  \\
&&-\frac{d}{dt} \sum_{n=1}^{+\infty} \int_{\omega_x}N_n^\epsilon\Big(V_{nn}^\epsilon-\underline{V_{nn}}\Big) dx +\sum_{n=1}^{+\infty} \int_{\omega_x}N_n^\epsilon\partial_t (V_{nn}^\epsilon-\underline{V_{nn}}) dx.
\end{eqnarray*}
Integrating by parts, the first right hand side term gives
\begin{equation*}
\int_{\omega_x} \partial_x \big(De^{-V_s^\epsilon}\partial_xu^\epsilon\big) \ln \Big(\frac{u^\epsilon}{\underline{u}}\Big) dx=-\int_{\omega_x} D e^{-V_s^\epsilon} \frac{(\partial_x u^\epsilon)^2}{u^\epsilon}dx+\int_{\omega_x} D e^{-V_s^\epsilon} \frac{\partial_x u^\epsilon \partial_x \underline{u}}{\underline{u}}dx.
\end{equation*}
Using the definition (\ref{v_nn_reg}) of $V_{nn}^\epsilon$, the last right hand side term gives
\begin{equation*}
\sum_{n=1}^{+\infty} \int_{\omega_x}N_n^\epsilon\partial_t  (V_{nn}^\epsilon-\underline{V_{nn}}) dx 
= \sum_{n=1}^{+\infty} \int_{\omega_x}N_n^\epsilon \partial_t < \mathcal{R}^{\epsilon}[V^\epsilon] - \mathcal{R}^{\epsilon}[\underline{V}], g_{nn}> dx.
\end{equation*}
At this point, the linearity and the selfadjointness of the regularization 
operator $\mathcal{R}^{\epsilon}$ and the regularized Poisson equation 
(\ref{pois_reg}) imply
\begin{equation*}
\sum_{n=1}^{+\infty} \int_{\Omega} N_n^\epsilon \partial_t \mathcal{R}^{\epsilon}[V^\epsilon-\underline{V}] g_{nn} dxdz = \frac{1}{2}\frac{d}{dt} \int_{\Omega}|\nabla (V^\epsilon-\underline{V})|^2 dxdz.
\end{equation*}
In the same way, we can write
\begin{equation*}
\frac{d}{dt} \sum_{n=1}^{+\infty} \int_{\omega_x}N_n^\epsilon (V_{nn}^\epsilon -\underline{V_{nn}}) dx = \frac{d}{dt} \int_{\Omega}|\nabla (V^\epsilon-\underline{V})|^2 dxdz.
\end{equation*}

Thus, defining $W$ as in (\ref{def_entropie}), we finally find
\begin{equation}
\frac{d W}{dt}=-\int_{\omega_x} D e^{-V_s^\epsilon} \frac{(\partial_x u^\epsilon)^2}{u^\epsilon}dx+\int_{\omega_x} D e^{-V_s^\epsilon} \frac{\partial_x u^\epsilon \partial_x \underline{u}}{\underline{u}}dx.
\end{equation}
We denote
\begin{equation}\label{entropy_diff_rate}
\mathcal{D}^{\epsilon}(t)=\int_{\omega_x} D e^{-V_s^\epsilon} \frac{(\partial_x u^\epsilon)^2}{u^\epsilon}dx
\end{equation}
the term which can be seen as an entropy dissipation rate. We also define $\beta=\|\partial_x \underline{u} / \underline{u}\|_{L^{\infty}(\omega_x)}$, $\beta < +\infty$. Consequently,
\begin{equation*}
\frac{d W}{dt} +\mathcal{D}^{\epsilon} \leq \beta \|D e^{-V_s^\epsilon} \partial_x u^\epsilon \|_{L^1(\omega_x)} \leq \beta \sqrt{\mathcal{D}^{\epsilon}} \sqrt{D\|N_s^\epsilon\|_{L^1(\omega_x)}}.
\end{equation*}
Using the inequality $2ab\leq \kappa^2 a^2 + \frac{b^2}{\kappa^2}$ for $\kappa>0$ small enough and Assumption \ref{assumption_D}, we get,
\begin{equation}\label{dW_dt}
\frac{d W}{dt} \leq \frac{d W}{dt} +C_1 \mathcal{D}^{\epsilon} \leq C_2 \|N_s^\epsilon\|_{L^1(\omega_x)},
\end{equation}
where $C_1$ and $C_2$ are two nonnegative constants.
Finally, using the inequality $\ln(x)-x+1\geq x+(1-e)$, for $x>0$, we have
\begin{equation*}
W \geq \sum_{n=1}^{+\infty} \int_{\omega_x} \underline{N_n} \Big(\frac{N_n^\epsilon}{\underline{N_n}}+1-e\Big) dx \geq \int_{\omega_x} N_s^\epsilon dx -(e-1) \int_{\omega_x} \underline{N_s} dx.
\end{equation*}
With \eqref{dW_dt}, it leads to
$$
 \frac{d W}{dt} \leq C_2 \|N_s^\epsilon\|_{L^1(\omega_x)}\leq C(W+C_0),
$$
where $C$ and $C_0$ are two nonnegative constants. We conclude thanks to a Gronwall's inequality and the fact that Assumption
\ref{assumption_init_N} and $V\in H^1(\Omega)$ imply that the initial 
entropy $W(0)$ is bounded. Moreover, we get the bound on the mass
\beq\label{mass}
\forall \, t\in [0,T],\qquad \int_{\omega_x} N_s^\epsilon\,dx \leq C.
\eeq
\end{proof}

\begin{corollary}\label{prop_dissipation} 
Let $T>0$ and $\epsilon\in[0,1]$. Under assumptions of Proposition \ref{prop_entropie1},
there exist $C_1$ and $C_2$ two nonnegative constants such that
\begin{equation}\label{a_priori_1}
\forall t \in [0,T],\hspace{1.5cm} \int_0^t \int_{\omega_x} |\partial_x \sqrt{N_s^\epsilon}|^2 dx ds \leq C_1,
\end{equation}
\begin{equation}\label{a_priori_2}
\forall t \in [0,T], \ \forall p\in [1,+\infty)\hspace{1.5cm} \int_0^t \| N_s^\epsilon \|_{L^{p}(\omega_x)} ds \leq C_2.
\end{equation}
\end{corollary}

\begin{proof}
In this proof, the letter $C$ is used to denote nonnegative constants. 
We can express the coefficient $\mathcal{D}^{\epsilon}$ defined in (\ref{entropy_diff_rate})
in terms of $N_s^\epsilon$ and $V_s^\epsilon$
$$
\mathcal{D}^{\epsilon}(t)=\int_{\omega_x} D \Big( 4|\partial_x \sqrt{N_s^\epsilon}|^2 + 2\partial_x N_s^\epsilon \partial_x V_s^\epsilon + N_s^\epsilon |\partial_x V_s^\epsilon|^2\Big) dx.
$$
After an integration by parts on the second term of the right hand side, 
we deduce
\begin{equation}\label{bornH1Ns}
4\|\partial_x \sqrt{N_s^\epsilon}\|^2_{L^2(\omega_x)} \leq C\Bigg( \mathcal{D}^{\epsilon}(t) +2 \int_{\omega_x} N_s^\epsilon \partial_{xx} V_s^\epsilon dx -2 \Big(N_s^\epsilon(L) \partial_{x} V_s^\epsilon(L) + N_s^\epsilon(0) \partial_{x} V_s^\epsilon(0)\Big)\Bigg).
\end{equation}
On the one hand, after calculations, we find
\begin{equation*}
\partial_{xx} V_s^\epsilon = \frac{\sum_ne^{-\big(E_n+V_{nn}^\epsilon\big)} 
\partial_{xx} V_{nn}^\epsilon}{\mathcal{Z}^\epsilon} +
\Bigg(\frac{\sum_ne^{-\big(E_n+V_{nn}^\epsilon\big)} \partial_{x} V_{nn}^\epsilon}{\mathcal{Z}^\epsilon}\Bigg)^2 
- \frac{\sum_ne^{-\big(E_n+V_{nn}^\epsilon\big)} (\partial_{x} V_{nn}^\epsilon)^2}{\mathcal{Z}^\epsilon}.
\end{equation*}
By the Cauchy-Schwarz inequality, the sum of the last two terms is nonpositive.
Moreover, from the regularized Poisson equation \eqref{pois_reg}
\begin{eqnarray*}
\partial_{xx} V_{nn}^\epsilon = <\partial_{xx} \mathcal{R}^{\epsilon}[V^\epsilon],g_{nn}> &=& <-\Delta_{z} \mathcal{R}^{\epsilon}[V^\epsilon],g_{nn}> - < \mathcal{R}^{\epsilon}[\rho^\epsilon], g_{nn}>  \vspace{0.5cm} \\
&\leq& \| \mathcal{R}^{\epsilon}[V^\epsilon] \|_{L^2(\omega_z)} \|\chi_n\|_{L^{\infty}(\Omega)} \|\chi_n\|_{H^2(\Omega)}.
\end{eqnarray*}
To obtain the last inequality, we remark that the second term is nonpositive and to treat the first term, we make an integration by parts and we use the fact that
\begin{equation*}
\| \Delta_z g_{nn} \|_{L^2(\omega_z)} \leq 2 \|\chi_n\|_{L^{\infty}(\Omega)} \|\chi_n\|_{H^2(\Omega)}.
\end{equation*}
Using the property (i) of Lemma \ref{lemme_reg} and (\ref{sumEn}), we conclude that
\begin{equation}\label{termxx}
\int_{\omega_x} N_s^\epsilon \partial_{xx} V_s^\epsilon dx \leq C  
\|N_s^\epsilon\|_{L^2(\omega_x)} \| V^\epsilon \|_{L^2(\Omega)}.
\end{equation}
On the other hand, we have
\begin{equation*}
N_s^\epsilon \partial_x V_s^\epsilon\,_{|\partial \omega_x} = 
N_b \frac{\sum_n <\partial_x \mathcal{R}^{\epsilon}[V^\epsilon],g_{nn}> 
e^{-(E_n+V_{nn}^\epsilon)}}{\sum_n e^{-(E_n+V_{nn}^\epsilon)}} \leq C N_b \int_{\omega_z} |\partial_x \mathcal{R}^{\epsilon}[V^\epsilon]_{|\partial \omega_x} | dz.
\end{equation*}
Thanks to the trace Theorem and Lemma \ref{lemme_reg}, we obtain
\begin{equation*}
N_s^{\epsilon} \partial_x V_s^{\epsilon} |_{\partial \omega_x} \leq C N_b \|V^{\epsilon}\|_{H^2(\Omega)}\leq C N_b \|N_s^{\epsilon}\|_{L^2(\omega_x)},
\end{equation*}
where we use the elliptic regularity and Lemma \ref{proposition_S} and 
\ref{lemme_reg} for the last inequality. With \eqref{bornH1Ns}
and \eqref{termxx}, we conclude that
\begin{equation*}
4\|\partial_x \sqrt{N_s^\epsilon}\|^2_{L^2(\omega_x)} \leq C \Big( \mathcal{D}^{\epsilon}(t) + \|N_s^\epsilon\|_{L^2(\omega_x)} \Big).
\end{equation*}
Applying the Gagliardo-Nirenberg inequality to the function $\sqrt{N_s^{\epsilon}}$ 
and using the bound on $\|N_s^\epsilon\|_{L^1(\omega_x)}$ (\ref{mass}), 
we obtain
\begin{equation*}
\forall t \in [0,T],\hspace{1.5cm} \int_{\omega_x} |\partial_x \sqrt{N_s^{\epsilon}}|^2 dx \leq C (1+ \mathcal{D}^{\epsilon}(t) ).
\end{equation*}
With (\ref{dW_dt}), we can say that $\int_0^t \mathcal{D}^{\epsilon}(s)ds \leq C(W(0) + \mathcal{N}_I t)$ for all $t\in[0,T]$ and consequently we obtain (\ref{a_priori_1}). 
Finally, (\ref{a_priori_2}) is a consequence of \eqref{a_priori_1}
with the Gagliardo-Nirenberg inequality.
\end{proof}

\subsection{Analysis of the regularized Nanowire Poisson system}
In this section, the surface density $N_s$ is assumed to be given and we only consider the resolution of the regularized Nanowire Poisson equation \eqref{pois_reg} with boundary conditions \eqref{bdyreg2}--\eqref{bdyreg3} for $\epsilon \in [0,1]$.

We introduce the functional space
\begin{equation*}
H^1_{\omega_x}=\{ V\in H^1(\Omega),\ \mbox{ s. t. } 
\forall x\in\partial\omega_x,\ \forall z\in\omega_z, \ V(x,z)=0\}.
\end{equation*}
Let us also take $V_0\in H^1(\Omega)$ such that $V_0=V_b$ on $\partial\omega_x\times \omega_z$ and $\partial_z V_0(x,z)=0$ for all $z\in\partial\omega_z$. A possibility is to take $V_0=\underline{V}$.
Most of the results presented here can be obtained by a straightforward 
adaptation of \cite{ddsp,ddsplog}. Thus we will not detail the proofs.
We first state the following existence result~:

\begin{proposition}\label{exist_poisson_nano}
Let $T>0$ and $\epsilon \in [0,1]$. We assume $N_s \in L^\infty([0,T];L^1(\omega_x))$ such that $N_s \geq 0$ a.e. 
Then the regularized Nanowire Poisson equation
\eqref{pois_reg} with boundary conditions \eqref{bdyreg2}--\eqref{bdyreg3} 
admits a unique solution $V^{\epsilon} \in V_0+H^1_{\omega_x}$ with a bound independent of $\epsilon$.
\end{proposition}

\begin{proof}
Using the selfadjointness of the regularization operator, a weak solution 
of (\ref{pois_reg}) is a critical point in the space $V_0+H^1_{\omega_x}$
of the functional
\begin{eqnarray}
\lefteqn{J(V,N_s)}\hspace{1.5cm}
&=& J_0(V)+ J_1(V,N_s)\nonumber\\
&=& \frac{1}{2}\int_{\Omega} |\nabla V|^2dxdz + \int_{\omega_x} N_s \ln \Big( \sum_{n=1}^{+\infty} e^{-\big(E_n+<\mathcal{R}^\epsilon[V],g_{nn}>\big)} \Big) dx.
\end{eqnarray}
Following the proof of Proposition 3.1 in \cite{ddsplog} and using 
Lemma \ref{lemme_spec}, 
we show that $J$ is a continuous, convex and coercive functional on $V_0+H^1_{\omega_x}$.
Thus $J$ admits a unique minimizer $V$ and we have a bound on $V$ in $H^1$
only depending on the $L^1$ norm of $N_s$.
\end{proof}

Then, we have the following continuity result~:
\begin{proposition}\label{cont}
Let $T>0$ and $\epsilon \in [0,1]$. Assume $N_s$ and $\widetilde{N}_s$ are given in $L^{\infty}([0,T];L^1(\omega_x))$ such that $N_s \geq 0$ and $\widetilde{N}_s \geq 0$ a.e. 
Then, the corresponding solutions $V^\epsilon$ and $\widetilde{V}^\epsilon$ of the regularized Nanowire Poisson equation
\eqref{pois_reg} with boundary conditions \eqref{bdyreg2}--\eqref{bdyreg3} verify
\begin{equation}\label{continuity}
\forall t\in[0,T], \hspace{0.5cm} \|V^\epsilon - \widetilde{V}^\epsilon \|_{H^1(\Omega)} \leq C \|N_s -\widetilde{N}_s \|_{L^1(\omega_x)}.
\end{equation}
Moreover, if $N_s$ and $\widetilde{N}_s$ belongs to $L^{\infty}([0,T];L^2(\omega_x))$, we have
\begin{equation}\label{h2_estimate}
\forall t\in[0,T], \hspace{0.5cm} \|V^\epsilon - \widetilde{V}^\epsilon \|_{H^2(\Omega)} \leq C \|N_s -\widetilde{N}_s \|_{L^2(\omega_x)},
\end{equation}
where $C$ stands for a nonnegative constant not depending on $\epsilon$.
\end{proposition}

\begin{proof}
Multiplying the regularized Poisson equation (\ref{pois_reg})
by $V^\epsilon- \widetilde{V}^\epsilon$ and integrating, we obtain
\begin{equation*}
 \int_{\Omega} |\nabla (V^\epsilon-\widetilde{V}^\epsilon)|^2 dx dz
= \int_{\Omega}  \Big( \big(N_s-\widetilde{N}_s\big) \mathcal{S}^\epsilon[V^\epsilon] +\widetilde{N}_s \big(\mathcal{S}^\epsilon[V^\epsilon]-\mathcal{S}^\epsilon[\widetilde{V}^\epsilon]\big) \Big)\mathcal{R}^\epsilon [V^\epsilon - \widetilde{V}^\epsilon ]  dx dz.
\end{equation*}
Because the functional $V\mapsto \mathcal{S}[V]$ is decreasing with respect to $V$, term
$(\mathcal{S}^\epsilon[V^\epsilon]-\mathcal{S}^\epsilon[\widetilde{V}^\epsilon])\mathcal{R}^\epsilon [V^\epsilon - \widetilde{V}^\epsilon]$ 
is nonpositive. We deduce
\begin{eqnarray}\label{eq_cont}
\int_{\Omega} |\nabla (V^\epsilon-\widetilde{V}^\epsilon)|^2 dx dz
&\leq
& \|N_s -\widetilde{N}_s\|_{L^1(\omega_x)} \| <\mathcal{R}^\epsilon [V^\epsilon - \widetilde{V}^\epsilon], \mathcal{S}^\epsilon[V^\epsilon]> \|_{L^\infty(\omega_x)}.
\end{eqnarray}
Then, (i) of Lemma \ref{proposition_S} and a Cauchy-Schwarz inequality gives
\begin{equation*}
\int_{\Omega} |\nabla (V^\epsilon-\widetilde{V}^\epsilon)|^2 dx dz \leq C \|N_s -\widetilde{N}_s\|_{L^1(\omega_x)} \|\mathcal{R}^\epsilon [V^\epsilon - \widetilde{V}^\epsilon]\|_{L^2_z L^\infty_x(\Omega)}.
\end{equation*}
We use the property (i) of Lemma \ref{lemme_reg} and the embedding $H^1 \hookrightarrow L^2_z L^\infty_x$. We obtain
\begin{equation*}\label{eq_cont2}
\int_{\Omega} |\nabla (V^\epsilon-\widetilde{V}^\epsilon)|^2 dx dz \leq C \|N_s -\widetilde{N}_s\|_{L^1(\omega_x)} \|V^\epsilon - \widetilde{V}^\epsilon\|_{H^1(\Omega)}.
\end{equation*}
Finally, thanks to the Poincar\'e inequality, we get (\ref{continuity}).

For the $H^2$ estimate, we have
\begin{equation*}
-\Delta(V^\epsilon-\widetilde{V}^\epsilon) = \rho^\epsilon - \widetilde{\rho}^\epsilon =
\mathcal{R}^\epsilon \Big[\big(N_s-\widetilde{N}_s\big)\mathcal{S}^\epsilon[V^\epsilon] + \widetilde{N}_s\big(\mathcal{S}^\epsilon[V^\epsilon]-\mathcal{S}^\epsilon[\widetilde{V}^\epsilon]\big)\Big].
\end{equation*}
Then we bound the $L^2$ norm of the right hand side as above 
using the spectral properties in Section \ref{spectrum}.
We finally get the $H^2$ estimate (\ref{h2_estimate}) from the elliptic
regularity.

\end{proof}

Finally, a straightforward adaptation of Proposition 3.2 of \cite{ddsplog}
gives the following convergence result as $\epsilon$ goes to $0$.
\begin{proposition}\label{limit_poisson}
As $\epsilon \rightarrow 0$, the solution $V^{\epsilon}$ of the regularized Nanowire Poisson system converges, uniformly with respect to $N_s \in L^{\infty}([0,T];L^1(\omega_x))$ such that $N_s \geq 0$ a.e, to the solution $V$ of the unregularized problem in $L^{\infty}([0,T];H^1(\Omega))$.
\end{proposition}

\subsection{Existence of solutions for the regularized system}

\begin{proposition}\label{existence_reg}
Let $T>0$ and $\epsilon \in [0,1]$ be fixed. Then the regularized problem NDDP$_{\epsilon}$ admits a unique solution $(N^\epsilon_s, V^\epsilon)$ with $N^\epsilon_s \in C([0,T];L^2(\omega_x))\cap L^2([0,T];H^1(\omega_x))$ and $V^{\epsilon}\in L^{\infty}([0,T];H^1(\Omega))$.
\end{proposition}

\begin{proof} The proof of this result follows closely the proof 
of Theorem 1.2 of \cite{ddsp}. Thus we will not detail it and only give
the main steps.
The proof relies on a fixed point argument on the map $F$ defined by :\\
\textbf{Step 1 :} For a given $N_s^\epsilon\geq 0$, we solve the regularized 
Nanowire Poisson equation \eqref{pois_reg} with boundary conditions
\eqref{bdyreg2}--\eqref{bdyreg3} and we define 
$V_s^\epsilon$ by (\ref{vs_reg}) which belongs to $L^{\infty}([0,T];H^1(\omega_x))$.\\
\textbf{Step 2 :} The effective potential $V_s^\epsilon$ being known, we solve 
the following drift-diffusion equation for the unknown $\widehat{N}^\epsilon_s$ 
\begin{equation*}
\partial_t \widehat{N}^\epsilon_s -\partial_x\Big(D(\partial_x \widehat{N}^\epsilon_s + \widehat{N}^\epsilon_s \partial_x V^\epsilon_s) \Big)=0,
\end{equation*}
with the boundary condition $\widehat{N}_s^\epsilon|_{\partial_{\omega_x}}=N_b$ and the initial value $\widehat{N}_s^\epsilon(0,x)=N_s^0(x)$. 
The map $F$ is then defined after these two steps by $F(N_s^\epsilon)=\widehat{N}_s^\epsilon$.

Then we can prove that for $T$ small enough,
$F$ is a contraction on the space $M_{a,T}$ defined by
$M_{a,T}=\{n, \|n\|_T \leq a\}$ where the norm is 
\begin{equation}\label{norme_T}
\|n\|_T=\Big[ \max_{0\leq t \leq T} \|n(t)\|^2_{L^2(\omega_x)} + \int_0^T \|n(t)\|^2_{H^1(\omega_x)}dt\Big]^{1/2}.
\end{equation}
We have then constructed a unique solution on a small time interval $[0,T_0]$.
Using the a priori estimate, we can iterate this procedure 
to construct a solution on $[T_0,2 T_0]$ that extend the previous one.
We iterate this construction until covering the interval $[0,T]$.
\end{proof}

\subsection{Passing to the limit $\epsilon \rightarrow 0$}

We construct a solution of the non-regularized Nanowire drift-diffusion
Poisson system by passing to the limit $\epsilon\to 0$ in the regularization.
First, we recall a statement of an Aubin-Lions lemma \cite{aubin,lions}~:
\begin{lemma}\label{lemme_aubin}
Take $T>0$, $q\in(1,+\infty)$ and let $(f_n)_{n\in \mathbb{N}}$ be a bounded sequence of functions in
$L^q([0,T];H)$ where $H$ is a Banach space. 
If $(f_n)_{n\in \mathbb{N}}$ is bounded in $L^q([0,T];V)$ where $V$ is compactly embedded in $H$ and $\partial f_n/\partial t$ is bounded in $L^q([0,T];V')$ uniformly with respect to $n\in \mathbb{N}$, then, $(f_n)_{n\in \mathbb{N}}$ is relatively compact in $L^q([0,T];H)$.
\end{lemma}

\noindent {\bf Proof of Theorem \ref{main_theorem}.} 
We fix $T>0$. From Proposition \ref{existence_reg}, there exists $N_s^{\epsilon}$ and $V^{\epsilon}$ solution of the regularized system NDDP$_{\epsilon}$ with the initial data $N_s^{\epsilon,0}$. 
The bound on $\|N_s^\epsilon\|_{L^1(\omega_x)}$ \eqref{mass} and the dissipation estimate (Corollary \ref{prop_dissipation})
furnish a bound of $\sqrt{N_s^{\epsilon}}$ in $L^{\infty}([0,T];L^2(\omega_x))$
and in $L^2([0,T];H^1(\omega_x))$. 
Thus, $N_s^{\epsilon}$ is bounded uniformly with respect to $\epsilon$ in
$L^2([0,T];W^{1,1}(\omega_x))$ (since we have the equality $\partial_x N_s^{\epsilon} =2\sqrt{N_s^{\epsilon}} \partial_x \sqrt{N_s^{\epsilon}}$).
Next, using the Cauchy-Schwarz inequality and 
Assumption \ref{assumption_D}, we obtain
\begin{equation*}
\int_0^T \Big(\int_{\omega_x} |\partial_x N^{\epsilon}_s + N^{\epsilon}_s \partial_x V^{\epsilon}_s |dx \Big)^2 dt \leq C \int_0^T \mathcal{D}^{\epsilon}(t) dt
\end{equation*}
where $\mathcal{D}^{\epsilon}$ is the entropy dissipation rate defined in (\ref{entropy_diff_rate})
which is bounded in $L^1([0,T])$ uniformly with respect to $\epsilon$. 
From the drift-diffusion equation \eqref{driduf_reg}, 
we conclude that $\partial_t N_s^{\epsilon}$ is bounded in $L^2([0,T];W^{-1,1}(\omega_x))$ uniformly with respect to $\epsilon$.
Therefore, we can apply the Aubin Lemma \ref{lemme_aubin} for $q=2$, $H=L^1(\omega_x)$ and $V=W^{1,1}(\omega_x)$. 
There exists a subsequence (that we still denote abusively $N_s^{\epsilon}$) such that $N_s^{\epsilon}\rightarrow N_s$ strongly in $L^2([0,T];L^1(\omega_x))$.
Finally, for this function $N_s$, we solve the unregularized Nanowire Poisson system and construct $V$ such that $V\in L^{\infty}([0,T];H^1(\Omega))$ (Proposition \ref{exist_poisson_nano}) and $\lim_{\epsilon \rightarrow 0} \|V^{\epsilon} - V \|_{L^2([0,T];H^1(\Omega))} =0$ (thanks to Proposition \ref{limit_poisson}). 

The last step is to pass to the limit $\epsilon \rightarrow 0$ in the drift-diffusion equation. We have
\begin{eqnarray*}
\int_0^T \int_{\omega_x} N_s^{\epsilon} \partial_x V_s^{\epsilon} dx dt 
&\leq& C \| N_s^{\epsilon} \|_{L^1([0,T];L^2(\omega_x))} \| V^{\epsilon} \|_{L^{\infty}([0,T];H^1(\Omega))}.
\end{eqnarray*}
Corollary \ref{prop_dissipation} shows that $\| N_s^{\epsilon} \|_{L^1([0,T];L^2(\omega_x))}$ is bounded independently of $\epsilon$ and we conclude that there exists a nonnegative constant $C$ independent of $\epsilon$ such that
\begin{equation}
\int_0^T \int_{\omega_x} N_s^{\epsilon} \partial_x V_s^{\epsilon} dx dt \leq C.
\end{equation}
It gives a sense to the drift-diffusion equation when $\epsilon\rightarrow 0$. Finally, using (\ref{estimate_v_nn}), we immediately deduce that $V_{nn}^{\epsilon}\rightarrow V_{nn}$ in $L^2([0,T];H^1(\omega_x))$ and that
\begin{equation*}
\partial_x V_s^{\epsilon} = \frac{\sum_n \partial_x V_{nn}^{\epsilon} e^{-\big(E_n + V_{nn}^{\epsilon}\big)}}{\mathcal{Z}^{\epsilon}}
\end{equation*}
converges in $L^2([0,T]\times\omega_x)$. It is enough to prove that
\begin{equation*}
N_s^{\epsilon} \partial_x V_s^{\epsilon} \rightharpoonup N_s \partial_x V_s \hspace{0.5cm}\hbox{in}\hspace{0.5cm} \mathcal{D}'([0,T]\times\omega_x).
\end{equation*}
Thus, up to an extraction, $(N_s,V)$ is a solution of the NDDP system.
Moreover, by semicontinuity, we can pass in the limit in the a priori estimates
such that we still have the relative entropy estimation of 
Proposition \ref{prop_entropie1} for $(N_s,V)$.
\qed

\appendix

\begin{lemma}\label{lemme1_appendix1}
For all function $p \mapsto \gamma(p)$ polynomially increasing as well as all its derivative and for all $\phi\in\mathcal{C}^{\infty}([0,T]\times\RR)$ compactly supported, the function $ \psi=\gamma\phi$ can be taken as test function in the weak formulation (\ref{varia_B}) of the Boltzmann equation.
\end{lemma}

\begin{proof}
Let $p \mapsto \xi_R(p)$ such that $\xi_R \in \mathcal{D}([-R,R])$, $0\leq \xi_R\leq 1$, $|\partial_p \xi_R|\leq1$ and $\xi_R \rightarrow 1$ a.e. when $R \rightarrow + \infty$. We set $\psi_R = \phi \gamma \xi_R$, function with which we can write the weak formulation (\ref{varia_B}). To pass to the limit $R\rightarrow +\infty$, it suffices from a Lebesgue theorem that $\gamma f_n^{\eta} \in L^1_p(\mathbb{R})$ and $p\gamma f_n^{\eta} \in L^1_p(\mathbb{R})$ as well as for $\partial_p \gamma$. However, with the Cauchy-Schwarz inequality,
\begin{equation*}
\int_{\mathbb{R}} |\gamma f_n^{\eta}|(1+|p|)dp \leq \Bigg(\int_{\mathbb{R}} (1+|p|)^2 M_n(p) \gamma^2(p) dp\Bigg)^{1/2} \Bigg(\int_{\mathbb{R}} \frac{(f_n^{\eta})^2}{M_n(p)}dp \Bigg)^{1/2} < \infty,
\end{equation*}
because $\gamma$ is polynomially increasing.
\end{proof}

\begin{lemma}\label{lemme2_appendix1}
Let $\Theta$ be defined in (\ref{def_theta}). There exist nonnegative constants $C_0$, $C_1$ and $C_2$ such that $\forall (t,x)\in[0,T]\times\RR$ :
\begin{equation}\label{eq:appendix1}
C_1(1+|p|) \leq \Big| \frac{\Theta_n}{\mathcal{M}_n}\Big| \leq C_2 (1+|p|),
\end{equation}
\begin{equation}\label{eq:appendix2}
\Big| \frac{\partial_s \Theta_n}{\mathcal{M}_n}\Big| \leq C_0 (1+|p|^2) \hspace{0.5cm}\hbox{for s=t, x and p.}
\end{equation}
\end{lemma}

\begin{proof}
By the definition (\ref{def_theta}), we have
\begin{equation}\label{eq_lambda_theta}
\lambda_n \Theta_n=\mathcal{Q}_B^+(\Theta)_n+\frac{p}{m_n^*}\mathcal{M}_n.
\end{equation}
where we denote $\mathcal{Q}_B^+(\Theta)_n=\mathcal{M}_n \sum_{n'\geq1} \int \alpha_{n,n'} \Theta_{n'} dp'$ and $\lambda_n=\sum_{n'\geq1} \int \alpha_{n,n'} \mathcal{M}_{n'}dp'$.
Using Assumption \ref{assumption_alpha}, we immediately find $\alpha_1\leq \lambda_n\leq \alpha_2$. 
Applying a Cauchy-Schwarz inequality,
\begin{equation*}
|\mathcal{Q}_B^+(\Theta)_n| \leq \alpha_2 \mathcal{M}_n \sum_{n'=1}^{+\infty} \int |\Theta_{n'}(p')| dp' \leq \alpha_2 \mathcal{M}_n \Bigg( \sum_{n'=1}^{+\infty} \int \frac{\big(\Theta_{n'}(p')\big)^2}{\mathcal{M}_{n'}(p')} dp' \Bigg)^{1/2}.
\end{equation*}
Since  $\Theta \in L^2_{\mathcal{M}}$, (\ref{eq:appendix1}) follows 
directly from (\ref{eq_lambda_theta}) and we differentiate it to obtain \eqref{eq:appendix2}.
\end{proof}

\section*{Acknowledgments}

This work has been partially supported by the project "QUAntum TRAnsport In Nanostructures" funded by the Agence Nationale de la Recherche (France). Authors also thank warmly Paola Pietra for fruitful discussions.



\end{document}